%% file: mnt_report.tex
\documentclass[11pt]{article}
\usepackage{extsizes,natbib,amssymb,amsmath,amsthm,fullpage,hyperref,graphicx,listings}

\newcommand{\Exp}[1]{{\text{E}}[ \ensuremath{ #1 } ]  }

\newcommand{\Cov}[1]{{\text{Cov}}[ \ensuremath{ #1 } ]  }
\newcommand{\bl}[1]{{\mathbf #1}}
\newcommand{\bs}[1]{{\boldsymbol #1}}
\newcommand{\tr}{\text{tr}}

\newtheorem{lemma}{Lemma}

\newtheorem{theorem}{Theorem}
\newtheorem{corollary}{Corollary}

\title{
Limitations on  detecting  row covariance in the presence of 
column covariance
}
\author{Peter D. Hoff  \\
Departments of Statistics and Biostatistics \\
University of Washington}
\date{\today}

\begin{document} 

\maketitle

\input{mnt_content}


\bibliography{mnt_report} 

\end{document}

%% file: mnt_content.tex
\begin{abstract}
Many inference techniques for  multivariate data analysis  
assume that the 
rows of the data matrix are realizations of independent and identically distributed
random vectors.
Such an assumption will be met, for example, if the rows  
of the data matrix 
are 
multivariate measurements on a set of 
independently sampled units.  
In the absence of an independent random sample, a relevant question is 
whether or not a statistical model 
that 
assumes such row exchangeability is plausible. 
One method for assessing this plausibility is 
a statistical test of row covariation.  
Maintenance of a constant type I error rate 
regardless of the column covariance  or matrix mean
can be accomplished with a test that is
invariant 
under an appropriate group of transformations. 
In the context of a class of  elliptically contoured matrix regression 
models 
(such as matrix normal models), 
I show that  
there are no non-trivial invariant tests if the number of rows is 
not sufficiently larger than the number of columns. 
Furthermore, I show that even if the number of rows 
is large, there are 
no non-trivial invariant tests
that have power to detect arbitrary row covariance in the 
presence of arbitrary column covariance. 
However, we can construct biased tests that have power to detect 
certain types of row covariance that may be encountered
in practice.  

\end{abstract}

\section{Introduction}
A canonical statistical model for 
an observed data matrix $\bl Y\in \mathbb R^{n\times p}$ is 
that the rows of the matrix are i.i.d.\ realizations from 
a mean-$\bs \mu$ $p$-variate normal distribution with covariance 
$\Sigma$. 
We write this hypothesized model as 
\begin{equation*}
  \bl Y \sim N_{n\times p} (\bl 1 \bs \mu^T , \Sigma \otimes \bl I_n ),  
\end{equation*}
where $\bl 1$ is the $n$-vector of all ones and 
 ``$\otimes$'' is the Kronecker product.  
If the rows represent multivariate measurements on a simple random 
sample of $n$ units from a population, then 
the assumption of i.i.d.\ rows is a valid 
one (or nearly valid
  for a large finite population, in the case of sampling 
without replacement). 
However, in many analyses the units 
are obtained from a convenience sample rather than a random sample. 
We might then  want to entertain an alternative model 
for the data, such as
\begin{equation*}
 \bl Y \sim N_{n\times p} (\bl 1 \bs \mu^T , \Sigma \otimes  \Psi ),   
\end{equation*}
where $\Psi$ is an unknown $n\times n$ covariance matrix 
describing dependence 
and heteroscedasticity
among the rows of $\bl Y$. 
This alternative model is the so-called 
matrix normal model \citep{dawid_1981}. 
Letting $\bl y_{i}$ and $\bl y_{i'}$ be two rows of $\bl Y$, 
this model implies that 
$\Cov{\bl y_i,\bl y_{i'} }  =\psi_{i,i'} \Sigma$. 

Several parametric and nonparametric tests of row dependence  
in the presence of column dependence
were considered in  
\citet{efron_2009} for the case that $p>n$. 
The parametric tests were 
based on estimates $\hat \Psi$ of $\Psi$
in  the matrix normal model. 
Efron suggested that such tests appear to be promising, but 
suffer some deficiencies. In particular, 
the distribution of the proposed estimate $\hat \Psi$ of $\Psi$ 
depends on the unknown 
value of $\Sigma$, a phenomenon that Efron referred to as 
``leakage.'' 
Proceeding with a similar approach, 
\citet{muralidharan_2010} constructed a permutation invariant 
test 
using asymptotic approximations in the 
$p>n$ scenario. 
This
test is conservative, 
and has power that 
depends on both $\Sigma$ and $\Psi$, that is, it also experiences 
some leakage. 

These issues suggest the use of 
invariant tests which, having power that doesn't depend on
the parameters of the null model, are leakage-free.
In this article, we characterize  the invariant tests
of $H:\Psi=\bl I$ versus $K:\Psi\neq \bl I$ in 
matrix regression models that have a stochastic representation of 
the form 
\[
 \bl Y = \bl X \bl B^T + \Psi^{1/2} \bl Z \bl \Sigma^{1/2}, 
\] 
where $\bl X\in \mathbb R^{n\times q}$ is an observed regression matrix,
$( \bl B, \Sigma, \Psi)$ are unknown parameters, and 
$\bl Z$ is a random matrix. 
For notational simplicity 
the results  in this article are developed for Gaussian random matrices, 
but as will be discussed, the results  
hold 
for a more general class of 
elliptically contoured matrix distributions,  including
heavy-tailed and contaminated distributions
\citep{gupta_varga_1994}.

The results of this article are primarily negative, illustrating inherent 
limitations on our ability to detect 
arbitrary row covariance in the presence of arbitrary column covariance. 
In the next section, I show that 
if $n\leq  p+q$ then  there are no non-trivial invariant tests of $H$ versus $K$.
In Section 3
I show that 
if $n> p+q$ then there are no non-trivial unbiased invariant tests. 
The implication of these results is that, 
for these matrix regression models, there are no useful invariant tests for arbitrary 
row covariance in the presence of arbitrary column covariance. 
On the bright side, one can construct biased invariant tests that 
have power to detect certain types of row dependence  that may be 
of interest in practice. 
For example, 
in 
Section 4 I 
obtain the UMP invariant test 
in a submodel 
where the eigenvector 
structure of $\Psi$ is known. 
This result is used  in Section 5 to construct a test 
that has the ability to detect 
positive dependence among arbitrary pairs of rows. 
The use of this test is illustrated 
on several 
datasets. 
In 
Section 6 I show how the results of the other sections 
generalize to non-Gaussian models, and discuss some open questions.

\section{Invariant test statistics}

We are interested in testing $H:\Psi = \bl I$ versus $K:\Psi\neq \bl I$ in the 
matrix normal regression model 
\begin{equation}
\bl Y \sim N_{n\times p}(  \bl X \bl B^T , \Sigma\otimes \Psi)   \ , \ \ 
\bl B\in \mathbb R^{p\times q} \ , \ \ 
\Sigma \in \mathcal S^{+}_p  \ , \ \ 
\Psi \in \mathcal S^+_n , 
\label{eqn:mnm}
\end{equation} 
where $\bl X$ is a known
$n\times q$ matrix with rank $q<n$ and 
$\mathcal S_k^+$ denotes the space of $k\times k$ 
nonsingular covariance matrices. 
Let $\bl P=(\bl I- \bl X (\bl X^T\bl X)^{-1} \bl X^T)$ 
so that $\bl R \equiv \bl P \bl Y$  is the matrix of residuals 
corresponding to the  least-squares  estimate of $\bl B$. 
Then 
$\Exp{  {\bl R}  {\bl R}^T  | \bl B, \Sigma \otimes \Psi} = 
 \tr (\Sigma)\times  \bl P \Psi \bl P$, 
which suggests the use of   ${\bl R}{\bl R}^T$ to test
whether or not $\Psi=\bl I$.  
The problem with such an approach is that, as pointed out 
by  \citet{efron_2009},
the distribution of $\bl R \bl R^T$ 
will generally depend on the 
unknown value of $\Sigma$. 
If the distribution of a 
test statistic 
depends on $\Sigma$, then maintaining the level of the test 
for all $\Sigma$ without sacrificing power is difficult.

With this in mind, we would like to identify
test statistics 
whose distributions under  $H$ do not
depend on $\bl B$ or $\Sigma$. 
To do this, we first note that 
the model and testing problem are invariant under
the group $\mathcal G$  of
transformations $g$ of the form 
$g(\bl Y ) = \bl X \bl C^T +\bl Y \bl A^T$
for $\bl C \in \mathbb R^{p\times q}$ and nonsingular $\bl A\in \mathbb R^{p\times p}$: If $\bl Y \sim N_{n\times p}(  \bl X \bl B^T , \Sigma\otimes \Psi)$
then $g(\bl Y) \sim N_{n\times p}(  \bl X [ \bl A \bl B + \bl C]^T  , \bl A \Sigma \bl A^T\otimes \Psi)$, 
It follows that the group $\mathcal G$  induces a group 
$\bar{ \mathcal G}$  of transformations on the parameter space 
of the form 
$\bar g (\bl B, \Sigma\otimes \Psi) =
  (\bl A\bl B + \bl C , \bl A \Sigma \bl A^T \otimes \Psi)$. 
This group is transitive on the null parameter space, and so  
any statistic or test function $\phi$ that is invariant to $\mathcal G$,  
meaning that $\phi( g (\bl Y ) ) = \phi( \bl Y)$ 
for all  $g\in \mathcal G$, 
will have 
a distribution that 
does not depend on  $\bl B$  or $\Sigma$.  
In particular, if $\phi$ is invariant then 
$\Exp{\phi(\bl Y)| \bl B , \Sigma\otimes \bl I}$ is constant 
in $\bl B$ and $\Sigma$.

\subsection{Maximal invariant statistics}
Any invariant test function or statistic must depend on $\bl Y$ only through
a statistic that is maximal invariant, that is, an invariant
function $M$ of $\bl Y$ for which $M(\bl Y_1) = M(\bl Y)$
implies $\bl Y_1 = g (\bl Y)$ for some $g \in \mathcal G$.
Therefore, characterizing
the class of invariant tests requires that we find
a maximal invariant statistic (since all maximal invariant 
statistics are functions of each other, we only need to find one). 
One maximal invariant statistic in particular has an intuitive form: 
Let $\hat {\bl B}$ be the OLS estimator of $\bl B$, 
let $\hat \Sigma = (\bl Y - \bl X \hat{\bl B}^T)^T 
                       (\bl Y - \bl X \hat{\bl B}^T)/n$, 
and let $\hat \Sigma^{-}$ be the  inverse or 
Moore-Penrose inverse of $\hat \Sigma$, depending 
on whether or not $\hat \Sigma$ is full rank. 
As will be shown below, 
the $n\times n$ matrix given by 
$M(\bl Y) =(\bl Y - \bl X \hat{\bl B}^T) \hat\Sigma^{-} 
 (\bl Y - \bl X \hat{\bl B}^T)^T/n$
constitutes a maximal invariant 
statistic. 
This statistic can also be written as 
$M(\bl Y) \equiv \bl G(\bl R)  =\bl R (\bl R^T \bl R)^{-} \bl R^T$ where 
$\bl R \equiv \bl P\bl Y$ is the matrix of residuals from 
the OLS fit.  
This matrix-valued function $\bl G$ maps any $n\times p$ matrix
$\bl R$ of rank $r$  to an $n\times n$ idempotent matrix
that represents
the $r$-dimensional hyperplane in $\mathbb R^{n}$ that is
spanned by the columns of $\bl R$.   
The set of
$r$-dimensional hyperplanes in $\mathbb R^{n}$ is 
a Grassman manifold, and points in this Grassman manifold
can be parametrized by the set of $n\times n$ idempotent matrices
of rank $r$.
In the context of the matrix regression model,
$\bl G(\bl R)$ gives the hyperplane that contains the 
residual row variation of the data matrix $\bl Y$.

To show that $\bl R (\bl R^T \bl R)^{-} \bl R^T$ is maximal 
invariant, we begin with two lemmas:
\begin{lemma}
Let $\bl R\in \mathbb R^{n\times p}$ be a matrix with rank $r>0$, 
and let $\bl R = \bl U \bl D \bl V^T$ be the singular value decomposition (SVD) of 
$\bl R$, so that
$\bl U^T \bl U = \bl V^T \bl V = \bl I_r$, and 
$\bl D\in \mathbb R^{r\times r}$ is  a positive definite diagonal matrix. 
Then 
$\bl G(\bl R) = \bl U \bl U^T$. 
\label{lem:grass} 
\end{lemma} 
\begin{proof}
\begin{align*}
\bl G( \bl R) = \bl R (\bl R^T \bl R)^{-} \bl R^T & =  
 \bl U \bl D \bl V^T ( \bl V \bl D^2 \bl V^T)^{-} \bl V \bl D \bl U^T \\
&= \bl U \bl D \bl V^T \bl V \bl D^{-2} \bl V^T \bl V \bl D \bl U^T \\
&= \bl U \bl U^T. 
\end{align*}
\end{proof}
Note that we are using a reduced form of the SVD that does not include any 
zero singular values. This is different from some 
computing environments (such as {\sf R}) that return
$n\wedge p$ left singular vectors even if 
$r< n\wedge p$. 

\begin{lemma}
If $\bl G(\bl R) = \bl G(\bl R_1)$ then there exists a nonsingular
matrix $\bl A$ such that $\bl R_1 = \bl R \bl A^T$. 
\label{lem:mil}
\end{lemma} 
\begin{proof}
Let $\bl R = \bl U \bl D \bl V^T$ be the SVD of $\bl R$, and 
let $\bl U_1$ be the matrix of left singular vectors of $\bl R_1$. 
Then $\bl U \bl U^T = \bl U_1 \bl U_1^T$  by the assumption and  Lemma \ref{lem:grass}, and so 
\begin{align*}
\bl R_1   =  \bl U_1 \bl U_1^T \bl R_1 
         &=   \bl U \bl U^T  \bl R_1  \\   
         &=   \bl U \bl D \bl V^T \bl V \bl D^{-1}  \bl U^T \bl R_1 \\ 
         &= \bl R ( \bl V \bl D^{-1} \bl U^T \bl R_1 ) 
         \equiv  \bl R  ( \bl V \bl F ) ,
\end{align*}
where $\bl F = \bl D^{-1} \bl U^T \bl R_1$.  
The rank of $\bl F$ is the same as that of $\bl R$ and $\bl R_1$, 
say $r$. If $r=p$ then $\bl A^T = \bl V \bl F$ is nonsingular and the result 
follows. If $r<p$ then let $\bl V^{\perp}\in \mathbb R^{p\times(p-r)}$ 
be an orthonormal basis 
for the null space of $\bl V$. Let $\bl A^T = [ \bl V\  \bl V^{\perp} ]  
 [   \bl F^T \ \bl G^T ]^T = \bl V \bl F + \bl V^{\perp} \bl G$, where $\bl G$ is any $(p-r)\times p$ 
matrix such that $[ \bl F^T \  \bl G^T ]^T$ is of rank $p$. 
Then $\bl A^T$ is nonsingular and $\bl R \bl A^T = \bl R_1$. 
\end{proof}   

It is now easy to derive the main result of this section, 
that $\bl R(\bl R^T\bl R)^{-} \bl R^T$ is maximal invariant:
\begin{theorem}
Let $\bl X \in \mathbb R^{n\times q}$ be of rank $q<n$ and 
let $\bl P = \bl I -\bl X(\bl X^T \bl X)^{-1} \bl X^T $.
Let $\mathcal G$ be the group of transformations  
on $\mathbb R^{n\times p}$  of the form 
$g(\bl Y) = \bl X \bl C^T + \bl Y \bl A^T$
for $\bl C \in \mathbb R^{p\times q}$ and 
nonsingular $\bl A \in \mathbb R^{p\times p}$.  
Then $M(\bl Y)\equiv \bl G(\bl R) = \bl R (\bl R^T \bl R)^{-} \bl R^T$ is maximal 
invariant, where $\bl R = \bl P\bl Y$. 
\end{theorem}
\begin{proof}
If is straightforward to show that 
$M(\bl Y)$ is invariant. 
Recall that to show 
it is maximal invariant, we must show that  
if $M(\bl Y_1) = M(\bl Y)$, then there exists 
a $g\in \mathcal G$ such that $\bl Y_1= g(\bl Y)$, 
or equivalently, that there exist matrices 
$\bl C \in  \mathbb R^{p\times q}$ and
nonsingular $\bl A \in \mathbb R^{p\times p}$ 
such that $\bl Y_1 =  \bl X \bl C^T +  \bl Y \bl A^T$. 
To find such matrices, let 
 $\bl R = \bl P \bl Y$ and $\bl R_1 = \bl P \bl Y_1$. 
If $\bl G(\bl R_1) = \bl G(\bl R)$ then by  
Lemma \ref{lem:mil} 
we must have 
$\bl R_1 = \bl R \bl A^T$ for a nonsingular matrix $\bl A$.  
Writing $\bl R_1$ and $\bl R$ in terms of $\bl Y_1$ and $\bl Y$, we have
\begin{align*}
( \bl I - \bl X (\bl X^T\bl X)^{-1} \bl X^T) \bl Y_1  &= 
( \bl I - \bl X (\bl X^T\bl X)^{-1} \bl X^T) \bl Y \bl A^T \\
 \bl Y_1 &= \bl Y \bl A^T + 
    \bl X (\bl X^T\bl X)^{-1} \bl X^T (\bl Y_1 - \bl Y \bl A^T) \\
 &= \bl Y \bl A^T + \bl X \bl C^T, 
\end{align*}
where $\bl C^T = (\bl X^T\bl X)^{-1} \bl X^T(\bl Y_1 - \bl Y \bl A^T) $. 
\end{proof}

To summarize, 
any invariant test statistic or test function 
must depend on $\bl Y$ only through 
$\bl R(\bl R^T \bl R)^{-} \bl R^T$, or equivalently 
$\bl U \bl U^T$, where $\bl U\in \mathbb R^{n\times r}$ is the matrix of 
left singular vectors of the  
rank-$r$ residual matrix $\bl R$.  
While  $\bl U^T \bl U= \bl I_r$ regardless of $r$, we also 
have $\bl U \bl U^T = \bl I_n$ if $r=n$. In this case, 
the maximal invariant statistic is constant, as is any
other $\mathcal G$-invariant function, including any 
invariant test function or statistic. 
Of course, any test that is based on a constant test function or statistic
is practically useless, as it 
must have constant power equal to its level. 
This unfortunate case occurs 
when $n$ is too small relative to 
$p$ and $q$:
\begin{corollary}
If $n\leq p+q$ then any $\mathcal G$-invariant function of $\bl Y$ is constant, 
and as a result, 
any invariant level-$\alpha$ test $\phi$ of $H: \Psi =\bl I$ 
versus $K: \Psi \neq \bl I$ has power 
$\Exp{ \phi(\bl Y) | \bl B, \Sigma\otimes \Psi } = \alpha$
for all $\bl B$, $\Sigma$ and $\Psi$.  
\label{cor:npq} 
\end{corollary}
\begin{proof}

The idempotent matrix $\bl P$ has $n-q$ eigenvectors with eigenvalues of 1,
and $q$ eigenvectors with eigenvalues of zero.
Let
$\bl H\in \mathbb R^{(n-q)\times n}$ be the matrix with rows
equal to the  first $n-q$  eigenvectors of $\bl P$,
so that $\bl H^T\bl H = \bl P$ and $\bl H \bl H^T = \bl I_{n-q}$.
Letting $\tilde{\bl Y} =  \bl H \bl Y$, we have  
    $\bl H^T \tilde{\bl Y} = \bl H^T\bl H \bl Y = \bl P \bl Y = \bl R $, 
and $\tilde {\bl Y} $ and $\bl R$ are  of the same rank 
$r =  (n-q) \wedge p$ for full rank $\bl Y$. 
The maximal invariant statistic can thus be expressed 
\begin{align*}
\bl R(\bl R^T \bl R)^{-} \bl R^T &= 
 \bl H^T \tilde{\bl Y} ( \bl Y^T \bl P \bl Y  )^{-} \tilde{\bl Y}^T \bl H \\
&= \bl H^T  \left ( 
   \tilde{\bl Y} ( \tilde{\bl Y}^T \tilde{\bl Y} )^{-}  \tilde{\bl Y}^T
   \right ) \bl H \\
 &= \bl H^T  \left (  \tilde{\bl U} \tilde{\bl U}^T 
    \right ) \bl H , 
\end{align*}
where $\tilde{\bl U}$ is the $(n-q)\times r$ matrix of left  
singular vectors of $\tilde{\bl Y}$.   
We have $\tilde{\bl U} \tilde{\bl U}^T = \bl I_{n-q}$ 
for all full rank $\bl Y$ if $r=n-q$, 
which happens if $n-q \leq p$, that is, if 
$n\leq p+q$.  In this case, the maximal invariant 
statistic $\bl G(\bl R)$ takes on the constant value 
$\bl H^T\bl H=\bl P$ for all full rank $\bl Y\in \mathbb R^{n\times p}$, 
and so any test function must be constant almost surely, 
and have power equal to its level. 
\end{proof}

This result says that there are no invariant tests of 
$H$ versus $K$ in the ``$p$ bigger than $n$'' regime. We illustrate 
with two simple examples. 
\begin{description}
\item[Mean-zero model:]
Consider testing $H$ versus $K$ in the 
mean-zero matrix normal model, given by
$\bl Y \sim N_{n\times p}(\bl 0, \Sigma\otimes \Psi)$.  
In this case, a maximal invariant statistic is 
$\bl Y (\bl Y^T \bl Y)^{-} \bl Y$. 
This is equal to $\bl I_n$ a.e.\ if  $n\leq p$, 
and so a non-trivial invariant test can only exist if $n>p$. 
\item[Column-means model:]  
Consider testing $H$ versus $K$ in the  
column means model, given by
$\bl Y \sim N_{n\times p}(\bl 1 \bs \mu^T , \Sigma\otimes \Psi)$,
where $\bs \mu\in \mathbb R^{p}$ is a vector of column-specific means. 
In this case, 
$\bl P = (\bl I - \bl 1 \bl 1^T/n)$, and $\bl R$ is obtained 
by centering the columns of $\bl Y$.  
The maximal invariant statistic is 
equal to $\bl P$ a.e.\ if  $n\leq p+1$, and so 
$n$ must be at least $p+2$ for a non-trivial invariant 
test to exist. 
\end{description}

\subsection{Reduction to the mean-zero model}
In some of what follows, it will be less notationally cumbersome 
to work with an alternative maximal invariant statistic. 
Let $\tilde {\bl Y } = \bl H \bl Y$ as in the proof of  Corollary \ref{cor:npq}. 
In that proof
we saw that 
\begin{align*} 
\bl G(\bl R) \equiv \bl R(\bl R^T \bl R)^{-} \bl R^T 
&= \bl H^T  \left ( 
   \tilde{\bl Y} ( \tilde{\bl Y}^T \tilde{\bl Y} )^{-}  \tilde{\bl Y}^T
   \right ) \bl H \\
 &\equiv  \bl H^T  \bl G(\tilde{\bl Y}) 
     \bl H. 
\end{align*}
Note also that $\bl G(\tilde{\bl Y})= \bl H \bl G(\bl R) \bl H^T$, and so 
$\bl G(\bl R)$ and $\bl G(\tilde{\bl Y})$ are functions of each other. 
Therefore,  $\bl G(\tilde{\bl Y})$ is maximal invariant as well 
(here we are abusing notation somewhat, letting $\bl G$ denote the same operation 
on matrices of different dimensions). 

The advantage of using $\bl G(\tilde{\bl Y})$ is that 
doing so 
reduces the testing problem to the mean-zero case: 
If $\bl Y\sim N_{n\times p}(\bl X\bl B^{T} ,\Sigma\otimes \Psi)$ then 
$\bl H \bl Y \equiv\tilde {\bl Y} \sim  N_{(n-q)\times p} (\bl 0, \Sigma\otimes \tilde \Psi )$, where $\tilde \Psi = \bl H\Psi \bl H^T$. 
Also note that 
 $\tilde \Psi$ ranges over $\mathcal S_{n-q}^+$ 
as $\Psi$ ranges over $\mathcal S_{n}^+$, 
and 
that $\tilde \Psi\neq \bl I_{n-q}$ implies
$\Psi \neq \bl I_n$ (but not vice versa). 
The testing problem 
of  $H:\tilde \Psi =\bl I_{n-q}$ versus $K: \tilde \Psi \neq \bl I_{n-q}$ in the 
mean-zero model for $\tilde {\bl Y}$ is invariant under 
the group $\mathcal G_L$ of linear transformations
of the form  $g(\tilde{\bl Y}) = \tilde{\bl Y} \bl A^T$  for nonsingular 
$\bl A$, and the statistic 
$\bl G(\tilde {\bl Y})= \tilde {\bl Y}( \tilde {\bl Y}^T \tilde {\bl Y})^{-} \tilde {\bl Y}^T$ is maximal invariant. 
Therefore, every $\mathcal G$-invariant 
level-$\alpha$ test under model (\ref{eqn:mnm}) 
is equivalent to a $\mathcal G_L$-invariant 
level-$\alpha$
test under the mean 
zero model, and vice-versa. 
This equivalence can be helpful in identifying 
limitations of $\mathcal G$-invariant tests. For example, 
consider the column means model where 
$\bl Y\sim N_{n\times p}(\bl 1 \bs \mu^T , \Sigma\otimes \Psi)$. 
An invariant test of $H:\Psi=\bl I_n$ versus $K:\Psi\neq \bl I_n$ 
is equivalent to a test of $H:\tilde \Psi=\bl I_{n-1}$ versus $K:\tilde \Psi\neq \bl I_{n-1}$ in the mean-zero model. This implies that an exchangeable 
 row covariance 
 $\Psi = \bl I + \omega \bl 1 \bl 1^T$ is not detectable 
by a $\mathcal G$-invariant test, as  
 $\tilde \Psi = \bl H (  \bl I + \omega  \bl 1 \bl 1^T) \bl H^T = \bl I_{n-1}$. 
This limitation makes intuitive sense, as exchangeable row covariance is 
manifested by adding a common random normal $p$-vector to each row of 
the data matrix, the effect of which is confounded with that of the mean vector
$\bs \mu$. 

\subsection{Reduction of row effects models}
Many datasets exhibit across-row heterogeneity that 
we may wish to represent with a mean model for $\bl Y$.  For example, 
the possibility that some rows and some columns give consistently higher 
 or consistently lower responses than average 
can be represented with a 
row and column effects model 
$\bl Y \sim N_{n\times p}(  \bs \alpha \bl 1^T_p + \bl 1_n \bs \beta^T , 
 \Sigma\otimes \Psi)$,  where $\bs \alpha\in \mathbb R^n$ and $\bs \beta\in \mathbb R^{p}$ 
are unknown parameters.  
This model is a special case of a row and column regression model, 
\begin{equation}
\bl Y \sim N_{n\times p}(  \bl A \bl W^T + \bl X \bl B^T, 
   \Sigma\otimes \Psi)
\label{eqn:rce}, 
\end{equation}
where $\bl W \in \mathbb R^{p\times q_1}$ and 
$\bl X \in \mathbb R^{n\times q_2}$ are observed matrices of column 
and row regressors. 

This model 
is not invariant to any group of transformations that includes 
multiplication on the right by arbitrary non-singular $p\times p$ matrices, 
as such transformations result in a different mean model 
(a bilinear regression model). 
However, using the ideas of the previous subsection we can construct 
test statistics having distributions 
that do not depend on the parameters  $\bl A$, $\bl B$ and $\Sigma$ of the null model.
Let $\bl P_W =\bl I - \bl W(\bl W^T\bl W)^{-1} \bl W^T$, 
and let $\bl H_W\in \mathbb R^{(p-q_1)\times p}$ 
be such that 
$\bl H_W^T\bl H_W = \bl P_W$ and
 $\bl H_W \bl H_W^T = \bl I_{p-q_1}$. 
Then 
$\bl Y \bl H_W^T \sim N_{n\times (p-q_1) } ( \bl X \tilde{\bl B}^T , 
  \tilde \Sigma  \otimes \Psi)$, 
where $\tilde{\bl B} = \bl H_W \bl B$ and 
$\tilde{\Sigma}= \bl H_W \Sigma  \bl H_W^T$. 
As $\bl B$ and $\Sigma$ range over $\mathbb R^{p\times q_2 }$ and 
$\mathcal S_{p}^+$, 
$\tilde{\bl B} $ and $\tilde \Sigma$ range over 
$\mathbb R^{(p-q_1) \times q_2}$ and 
$\mathcal S_{p-q_1}^+$ respectively. 
In this way, we can reduce the model (\ref{eqn:rce}) to 
the model (\ref{eqn:mnm}) considered in previously. 
Defining $\bl P_X$ and $\bl H_X$ analogously to $\bl P_W$ and $\bl H_W$, 
we can define
$\bl R = \bl P_X \bl Y \bl P_W^T$ 
and use $\bl G(\bl R)$ to construct a test statistic 
whose distribution does not depend on the parameters in the null model. 
Also note that $\bl R$ can be expressed as
$\bl R =  \bl H_X^T  \bl H_X \bl Y \bl H_W^T \bl H_W \equiv
      \bl H_X^T \tilde{\bl Y} \bl H_W $, where
$\tilde{\bl Y } \sim N_{(n-q_2)\times(p-q_1) } (\bl 0 , \tilde\Sigma \otimes 
   \tilde \Psi)$.
Furthermore, we have
\begin{align*}
\bl G(\bl R)  =  \bl R \left (\bl R^T\bl R\right )^{-} \bl R^T &=  \bl H_X^T \tilde{\bl Y} \bl H_W   
   \left (  \bl H_W^T \tilde{\bl Y}^T \tilde{\bl Y}  \bl H_W \right )^{-} 
   \bl H_W^T \tilde{\bl Y}^T \bl H_X \\
&=\bl H_X^T \tilde{\bl Y}\bl H_W \bl H_W^T
\left (\tilde{\bl Y}^T \tilde{\bl Y}\right)^{-} \bl H_W \bl H_W^T  \tilde{\bl Y}^T \bl H_X \\ 
&= \bl H_X^T  \tilde{\bl Y} \left ( \tilde{\bl Y}^T \tilde{\bl Y}\right)^{-} 
   \tilde{\bl Y}^T \bl H_X = 
  \bl H_X^T  \bl G(\tilde{\bl Y}) \bl H_X, 
\end{align*}
and so $\bl G(\bl R)$ and $\bl G(\tilde{\bl Y})$ are functions
of each other.

The row and column regression model can therefore be reduced  to a  mean-zero
model, which is invariant under $\mathcal G_L$. Any $\mathcal G_L$-invariant
test of
$H:\Psi=\bl I_n$ versus $K:\Psi\neq \bl I_n$ based on the
residual matrix $\bl R$
corresponds to a $\mathcal G_L$-invariant test of $H:\tilde\Psi =\bl I_{n-q_2}$ versus
$K:\tilde\Psi\neq \bl I_{n-q_2}$ in the mean-zero model for
$\tilde{\bl Y}$, and vice versa.

\section{Invariant tests and bias}

Can an invariant test have non-trivial power for all values of 
$\Psi$? 
For notational simplicity we first 
answer this question for the mean-zero model
$\bl Y \sim N_{n\times p}(\bl 0 , \Sigma \otimes \Psi)$,
and then extend the result to the
matrix normal regression model (\ref{eqn:mnm}). 
As described 
above, 
the mean-zero  model is invariant under the group $\mathcal G_L$ 
of nonsingular linear transformations
$g( \bl Y) = \bl Y \bl A^T$, and this 
group is transitive on the null parameter space.  
We consider only the case that $n>p$, otherwise by 
Corollary \ref{cor:npq} the maximal invariant is constant 
and there are no non-trivial invariant tests. 
In this case of $n>p$, 
a maximal invariant statistic 
is $\bl G(\bl Y) = 
  \bl Y(\bl Y^T\bl Y)^{-1} \bl Y^T = \bl U \bl U^T$, where 
$ \bl U \bl D \bl V^T$ is the SVD of $\bl Y$, 
or  alternatively $\bl U = \bl Y (\bl Y^T\bl Y)^{-1/2}$. 
Note that although these values of $\bl U$ are in general different, 
they give the same value of $\bl U\bl U^T$. 

\subsection{Unbiased tests have trivial power}
The main result of this section is  negative:
There are no non-trivial unbiased invariant tests 
of $H:\Psi=\bl I$ versus $K:\Psi\neq \bl I$. Put another way, 
if $\phi$ is a function of $\bl U\bl U^T$ under the mean-zero 
matrix normal model, then 
it cannot satisfy $\Exp{ \phi  | \Sigma \otimes \Psi}
> \Exp{\phi | \Sigma\otimes \bl I}$ for all values of 
$\Psi\neq \bl I$. More specifically, we will prove the following result:
\begin{theorem}
Let $\phi:\bl Y \rightarrow [ 0,1 ]$ be a $\mathcal G_L$-invariant function
such that $\Exp{ \phi | \bl I \otimes  \bl I } =\alpha$.
If $\Exp{ \phi |  \Sigma \otimes  \bl E \Lambda \bl E^T  } \geq \alpha$
for a fixed positive definite diagonal matrix $\Lambda$ and
all $\bl E \in \mathcal O_{n}$, then
 $\Exp{ \phi |  \Sigma \otimes  \bl E \Lambda \bl E^T  } = \alpha$
for  
 all $\bl E\in \mathcal O_{n}$.
\label{prop:mainresult}
\end{theorem} 
\begin{proof}
If $\phi$ is $\mathcal G_L$-invariant it must be a function of the 
maximal invariant statistic 
$\bl U \bl U^T$. 
First consider the distribution of $\bl U\bl U^T$ when the covariance of 
$\bl Y$ is 
$ \Sigma \otimes \Psi$. Let $\bl E \Lambda \bl E^T$ be the eigendecomposition
of $\Psi$, let $\bl Z$ be an $n\times p$ random  matrix  with standard 
normal entries, and let
$\Sigma^{1/2}$ be the symmetric square root of $\Sigma$. Then $\bl Y \stackrel{d}{=} \bl E \Lambda^{1/2} \bl Z \Sigma^{1/2}$ and 
\begin{align*}
\bl U\bl U^T & = \bl Y (\bl Y^T \bl Y)^{-1} \bl Y^T \\ 
& \stackrel{d}{=}  
\bl E \left ( \Lambda^{1/2} \bl Z \Sigma^{1/2}( \Sigma^{1/2} \bl Z^T \Lambda^{1/2} \bl E^T \bl E \Lambda^{1/2} \bl Z\Sigma^{1/2})^{-1} \Sigma^{1/2} \bl Z^T \Lambda^{1/2}  \right )  \bl E^T   \\
&= 
\bl E \left ( \Lambda^{1/2} \bl Z (  \bl Z^T \Lambda \bl Z)^{-1}  \bl Z^T \Lambda^{1/2}  \right )  \bl E^T. 
\end{align*} 
Now let $\bl W = \bl Z (\bl Z^T \bl Z)^{-1/2}$, and note that $\bl W$ is uniformly 
distributed on the Stiefel manifold $\mathcal V_{p,n}$  
\citep[section 8.2]{gupta_nagar_2000}. 
A few additional 
steps show that
\begin{equation}
 \bl U\bl U^T \stackrel{d}{=} 
\bl E \left ( \Lambda^{1/2} \bl W ( \bl W^T \Lambda \bl W)^{-1} \bl W^T \Lambda^{1/2}  \right ) \bl E^T. 
\label{eqn:dmi}
\end{equation}
The term in  parentheses
is a random  $n\times n $ idempotent matrix 
and can be written as $\bl F\bl F^T$, where $\bl F$ is a random 
element of $\mathcal V_{p,n}$ with a distribution that depends 
on $\Lambda$ but not $\bl E$. 
Therefore, the maximal invariant statistic satisfies
$\bl U \bl U^T \stackrel{d}{=} \bl E \bl F \bl F^T \bl E^T $
where $\bl E$ is fixed and $\bl F$ is random but does not 
depend on $\bl E$. 

We now use this fact to show that, for any given $\Lambda$, 
no invariant level-$\alpha$ test can have non-trivial power for  all 
values of $\bl E$. In other words, if $\phi$ is a level-$\alpha$ 
invariant test then 
\[ 
\Exp{ \phi(\bl Y) | \Sigma\otimes \bl E \Lambda \bl E^T } \geq \alpha \ \forall \ \bl E \in \mathcal O_n  \ \ \text{implies}  \ \ 
\Exp{ \phi(\bl Y) |\Sigma  \otimes \bl E \Lambda \bl E^T } = \alpha \ \forall \  \bl E \in \mathcal O_n. 
\]
To see this, note that  
under the null hypothesis we have $\Lambda= \bl E = \bl I$
 and so  from (\ref{eqn:dmi}) we have 
$\bl U \bl U^T \stackrel{d}{=} \bl W \bl W^T$, 
where  $\bl W$ is uniformly
distributed on $\mathcal V_{p,n}$. 
Therefore, an invariant level-$\alpha$ test will be of the form 
$\phi(\bl Y) = f( \bl U \bl U^T)$, where 
$f$ satisfies 
$\Exp{ f(\bl W\bl W^T) } = \alpha$.

Now consider $\Lambda\neq \bl I$
and a uniform ``prior'' distribution 
for $\bl E$.
In this case 
the distribution of $\bl U\bl U^T$, conditional on $\Lambda$, is given 
by 
$\bl U \bl U^T   \stackrel{d}{=} 
 \bl E \bl F \bl F^T \bl E^T$
with $\bl E\sim$ uniform($\mathcal O_n$), $\bl F\in \mathcal V_{p,n}$ having 
the distribution depending on $\Lambda$ described above, 
and $\bl E$ and $\bl F$ being independent. 
By results of \citet[chap. 2]{chikuse_2003a}, 
the uniformity of $\bl E$ and the independence of $\bl E$ and $\bl F$ imply 
that 
\[ \bl U \bl U^T \stackrel{d}{=}  \bl W \bl W^T, \]
as is the case under the null distribution. 
In other words,
\[  
\int \Exp{ f(\bl U\bl U^T)| \Sigma \otimes \bl E \Lambda \bl E^T) }  
   \ \mu(d \, \bl E)  = 
  \Exp{ f(\bl W \bl W^T) } = \alpha, 
\]
where $\mu$ is the uniform probability measure over
$\mathcal O_n$. 
This implies that 
if the power $\Exp{ f(\bl U\bl U^T) | \Sigma \otimes \bl E \Lambda \bl E^T }$
is greater than $\alpha$ on a set of $\bl E$-values with $\mu$-measure greater than zero, 
it must be less than $\alpha$ on a set with non-zero measure as well. 
Equivalently, if 
$\Exp{ \phi(\bl Y) | \Sigma\otimes \bl E \Lambda \bl E^T } \geq \alpha$
for $\bl E$ almost everywhere $\mu$, then 
$\Exp{ \phi(\bl Y) | \Sigma\otimes \bl E \Lambda \bl E^T }  = \alpha$
for $\bl E$ almost everywhere $\mu$. 
Finally, continuity of the  power function implies that  
these relations that hold almost everywhere also hold 
everywhere on $\mathcal V_{p,n}$. 
\end{proof}

\subsection{Likelihood ratio tests}
One type of invariant test is a likelihood ratio test. By the above result, 
such a test must either be biased or have power equal to its level. 
Here we show that it is the latter. 
Negative two times the mean-zero matrix normal log likelihood is 
\[ 
- 2\log p(\bl Y| \Sigma\otimes \Psi)  = 
  p \log |\Psi | + n \log |\Sigma | + \tr(\Sigma^{-1} \bl Y \Psi^{-1} \bl Y^T)
 + c,
\]
where $c$ doesn't depend on $\bl Y$, $\Sigma$ or $\Psi$. 
For every positive definite $\Psi$, 
this is minimized  
in  $\Sigma$ by $\hat \Sigma = \bl Y^T \Psi^{-1} \bl Y/n$, 
giving
\begin{align} 
-2 \log p(\bl Y | \hat \Sigma \otimes \Psi) & =
  p \log |\Psi | + n \log |\bl Y^T \Psi^{-1} \bl Y/n| + np  + c  \nonumber \\
&  =   p \log |\Psi | + 
  n \log |\bl U^T \Psi^{-1} \bl U| + n\log|\bl D^2|   +  n(p-\log n )  + c 
\label{eqn:plik}
\end{align}
where  now $\bl U = \bl Y (\bl Y^T \bl Y)^{-1/2}$. 
Having a similar form are the densities for $\bl U$ and $\bl G=\bl U \bl U^T$
with respect  to the uniform probability measures
on the Stiefel and Grassman manifolds, respectively.
These densities, 
derived by \citet{chikuse_2003a}, give the following log-likelihoods:
\begin{align}
-2\log p_U(\bl U| \Psi) & =  p\log|\Psi| + n \log | \bl U^T \Psi^{-1} \bl U|
\label{eqn:ldu}  \\
-2\log p_G(\bl G|\Psi ) &=   p \log|\Psi| + n | \bl I - (\bl I-\Psi^{-1})\bl G|
 \label{eqn:ldg}, 
\end{align}
Some matrix manipulation shows that 
(\ref{eqn:ldg}) can be expressed as $-2 \log p_U( \tilde {\bl U} |\Psi)$ 
for any $\tilde{\bl U}$ such that $\tilde{\bl  U} \tilde {\bl U}^T =  
  \bl U\bl U^T =\bl G$. 

All three of these likelihoods depend on $\Psi$ only through 
$p\log|\Psi| + n \log | \bl U^T \Psi^{-1} \bl U|$. 
This term is unbounded below in $\Psi$, which can be seen as follows:
Parametrize $\Psi$ in terms of its eigendecomposition 
$\bl E \Lambda \bl E^T$, and let $\bl E = [ \bl U \ \bl U^{\bot } ]$, 
where $\bl U^{\bot }$ is the orthogonal complement of $\bl U$. 
Then $p\log|\Psi| + n \log | \bl U^T \Psi^{-1} \bl U| 
 =  - n \sum_{j=1}^p \lambda_ j + p \sum_{j=1}^n \log \lambda_j$, 
which approaches $-\infty$ as any of $\lambda_{p+1},\ldots, \lambda_n$ 
approach zero. Alternatively, 
\begin{align*}
p\log|\Psi| + n \log | \bl U^T \Psi^{-1} \bl U| &= 
   - n \sum_{j=1}^p \lambda_ j + p \sum_{j=1}^n \log \lambda_j \\
 &= - \left [(n-p) \sum_{j=1}^p \log \lambda_j  - p\sum_{j=p+1}^n \log \lambda_j \right ] \\
 & \leq  - \left [ (n-p)\log \lambda_1 + \sum_{j=2}^p \log \lambda_p  - p\sum_{j=p+1}^n \log \lambda_p \right  ] \\
&= -(n-p) \log (\lambda_1/\lambda_p), 
\end{align*}  
and so the likelihood is also unbounded in any submodel for $\Psi$ in 
which the
first eigenvalue may be made arbitrarily larger than the $p$th 
eigenvalue. 
As a result, 
all three 
likelihoods are unbounded 
in $\Psi$, and so the likelihood ratio statistic 
is constant (infinity). Therefore, 
the only way that a 
likelihood ratio test can have  level $\alpha\in (0,1)$ is if it is 
equal to the randomized test $\phi(\bl Y)=\alpha$. 

\subsection{Matrix normal regression model}
Finally, we 
apply the result in Theorem \ref{prop:mainresult} 
to the problem of testing for row dependence 
in the matrix normal regression  model (\ref{eqn:mnm}):
\begin{corollary} 
Let  $\phi$ be a 
level-$\alpha$
$\mathcal G$-invariant test of 
$H:\Psi=\bl I$ versus $K:\Psi\neq \bl I$ in the 
 model $\bl Y \sim N_{n\times p}(\bl X\bl B^T, 
\Sigma\otimes\Psi)$. 
If $\Exp{ \phi(\bl Y) |\bl B,  \Sigma\otimes \Psi }\geq \alpha$ 
for all $\Psi \in \mathcal S_{n}^+$ then 
 $\Exp{ \phi(\bl Y) |\bl B,  \Sigma\otimes \Psi }= \alpha$ 
 for all $\Psi \in \mathcal S_{n}^+$. 
\end{corollary}
\begin{proof} 
Recall from Section 2 that 
such a test function 
can be expressed as $\phi(\bl Y) =  f( \bl H \bl Y)$ for 
$\bl H \in \mathbb R^{(n-q)\times n}$ satisfying 
$\bl H^T \bl H= \bl I - \bl X (\bl X^T\bl X)^{-1} \bl X^T$. 
Now let $\tilde \phi(\tilde{\bl Y}) = f( \tilde {\bl Y})$ for 
$\tilde{\bl Y} \in \mathbb R^{(n-q)\times p}$.  
Then 
\begin{align*}
\Exp{ \phi(\bl Y) |  \Sigma\otimes\Psi) }
    &=           \Exp{ f(\bl H \bl Y) |  \Sigma\otimes \Psi) } \\
    &= \Exp{ \tilde \phi(\tilde {\bl Y}) | \Sigma \otimes \tilde \Psi) }
\end{align*} 
where $\tilde {\bl Y} \sim N_{(n-q)\times p}(\bl 0, 
\Sigma\otimes \tilde  \Psi)$, with 
$\tilde \Psi = \bl H \Psi \bl H^T$. 
Plugging in $\Psi=\bl I_{n}$ shows that  
$\tilde \phi$ is a level-$\alpha$ 
$\mathcal G_L$-invariant test 
of $\tilde H: \tilde   \Psi = \bl I_{n-q}$ versus 
 $\tilde K: \tilde   \Psi \neq  \bl I_{n-q}$
 for 
the model 
$\tilde {\bl Y} \sim N_{(n-q)\times p}(\bl 0, 
\Sigma\otimes \tilde  \Psi)$. 
The conditions of the corollary imply that 
$\Exp{ \tilde \phi(\tilde {\bl Y} |
 \Sigma \otimes \tilde \Psi} \geq \alpha$ for 
all $\tilde \Psi \in \mathcal S_{n-q}^+$, and so Theorem
\ref{prop:mainresult}
implies that 
$\Exp{ \tilde \phi(\tilde{\bl Y}) | \Sigma \otimes \tilde \Psi} = \alpha$  
for 
all $\tilde \Psi \in \mathcal S_{n-q}^+$. 
Since the power of $\phi$ under any $\Psi$ is equal to the power 
of $\tilde \phi$ under some $\tilde \Psi$, 
we have that 
$\Exp{ \phi(\bl Y) |  \Sigma\otimes\Psi) }=\alpha$ for all  
$\Psi \in \mathcal S_{n}^+$. 
\end{proof}

\section{UMP tests in spiked covariance submodels} 
The absence of unbiased tests with non-trivial power 
under all alternatives $\Psi \in \mathcal S_n^+$ 
indicates that any useful tests of row dependence must 
focus on particular types of dependence. For example, 
if the rows of $\bl Y$ are measured at different times 
or locations, it makes sense to test for dependence using 
a spatial or temporal autoregressive submodel 
for $\Psi$. This can be done, for example, with a likelihood 
ratio test based on the likelihoods 
(\ref{eqn:plik}), (\ref{eqn:ldu}) or (\ref{eqn:ldg})
restricted to a subset of $\Psi$-values. 
Simulation results (not presented here) suggest that such 
tests perform reasonably well.

Another popular submodel of 
$\mathcal S_n^+$ are 
the so-called ``spiked covariance'' 
or partial isotropy models \citep[section 8.4]{mardia_kent_bibby_1979},
where $\Psi$ takes the form 
$\Psi = \bl C \Omega \bl  C^T + \bl I$ with
$\bl C \in \mathcal V_{r,n}\subset \mathbb R^{n\times p}$ and 
$\Omega\in \mathbb R^{r\times r}$ is a positive definite diagonal matrix. 
The eigenvalues of such a covariance matrix are 
$(\omega_1+1,\ldots, \omega_r+1,1, \ldots, 1) \in \mathbb R^n$, 
and the eigenvectors can be taken as $\bl E = [ \bl C \,  \bl C^\bot ]$, 
where $\bl C^{\bot}\in \mathcal V_{n-r,n}$ satisfies 
 $\bl C^T \bl C^{\bot} = \bl 0$. 
As described in the previous section, any level-$\alpha$ 
test that has power greater 
than $\alpha$  on a non-empty set of 
$\bl E$-values (and hence a non-empty set of $\bl C$ values)
must be biased.
Therefore, any submodel for which we have a useful test 
must restrict the eigenvectors of $\Psi$ in some way.

Perhaps the simplest case  of such a restricted submodel 
is a rank-1 spiked covariance model
of the form $\Psi =  \omega \bl c \bl c^T + \bl I $, where 
$\omega \in \mathbb R^+$ is unknown and $\bl c$ is a known 
unit vector in $\mathbb R^n$.  
In this case, a best invariant test of $H:\omega=0$ versus 
$K: \omega>0$ can be identified and described. 
As in the last section, we begin with the mean-zero model 
and then extend the result to the more general case. 
\citet{chikuse_2003a} shows that 
the density of $\bl U = \bl Y (\bl Y^T\bl Y)^{-1/2}$
for $\bl Y \sim N_{n\times p}(\bl 0,\bl I\otimes \Psi)$ 
is in general given by 
$p(\bl U | \Psi) = 
 |\Psi|^{-p/2} | \bl U^T \Psi^{-1} \bl U|^{-n/2}$. 
For $\Psi =  \omega \bl c \bl c^T + \bl I$, this reduces to 
\[
 p(\bl U |\omega, \bl c ) = (1+\omega)^{-p/2} (1-\bl c^T \bl U \bl U^T\bl c \omega/
 (1+\omega))^{-n/2} .
\]
It is easily checked that this class of densities has a monotone likelihood 
ratio in the statistic $t_{\bl c}(\bl U) = \bl c^T \bl U \bl U^T\bl c$, 
and so a uniformly most powerful test of 
$H: \omega=0$ versus $K:\omega>0$ is given by rejecting 
$H$ when $t_{\bl c}(\bl U)$ is large. 
Since such a test is UMP among tests based on $\bl U$ and is a 
function of the maximal invariant statistic $\bl U\bl U^T$, 
it is also  the uniformly most powerful invariant test for its 
level. 
Furthermore, the distribution of this test statistic can 
be obtained under both the null and alternative hypotheses. 
Using the result from  (\ref{eqn:dmi}), the test statistic can 
be written as 
\[
\bl c^T \bl U\bl U^T \bl c \stackrel{d}{=} 
\bl c^T \bl E  \Lambda^{1/2} \left (  \bl W ( \bl W^T \Lambda \bl W)^{-1} \bl W^T  \right ) \Lambda^{1/2}  \bl E^T \bl c, 
\]
where $\bl W$ is uniform on $\mathcal V_{p,n}$, and 
$\bl E$ and $\Lambda$ are the eigenvector and 
eigenvalue matrices of $\Psi$. For the rank-1 spiked model, 
we have $\bl E^T \bl c = (1,0,\ldots,0)^T \equiv \bl e_1^T$
and so 
\[ 
\bl c^T \bl U\bl U^T \bl c \stackrel{d}{=}  
  (\omega+1) \left(  \bl  W ( \bl W^T \Lambda \bl W)^{-1} \bl W^T  \right )_{[1,1]}. 
\]
In this case where  $\Lambda = \bl I + \omega \bl e_1 \bl e_1^T$, 
we have 
\begin{align*}
  \bl W^T \Lambda \bl W  &=  \bl I + \omega 
 \bl W^T\bl e_1 \bl e_1^T \bl W \\ 
 ( \bl W^T \Lambda \bl W)^{-1}&=   
 \bl I  - \bl w_1 \bl w_1^T  \frac{\omega}{1+ \omega |\bl w_1|^2 } ,
\end{align*}
where $\bl w_1\in \mathbb R^{p}$ is the first row of $\bl W$. 
We then have 
\begin{align*}
 \bl  W ( \bl W^T \Lambda \bl W)^{-1} \bl W^T &= 
  \bl W\bl W^T  - \bl W \bl w_1 \bl ( \bl W \bl w_1)^T \frac{\omega}{1+ \omega |\bl w_1|^2 }   \\
\left(  \bl  W ( \bl W^T \Lambda \bl W)^{-1} \bl W^T \right)_{[1,1]} &= 
   |\bl w_1|^2 - |\bl w_1|^4 \frac{\omega}{1+ \omega |\bl w_1|^2 }. 
\end{align*}
Letting $b =|\bl w_1|^2$, rearranging gives 
\[
 t_{\bl c}(\bl U) = \bl c^T \bl U\bl U^T \bl c 
 \stackrel{d}{=} \frac{1+\omega}{1+ b \omega } b.  
\]
Note that 
the right-hand side 
is an increasing function in $\omega$ for each fixed $b$, 
and so  the distributions of 
$t_{\bl c}(\bl U)$ are stochastically increasing in $\omega$. 
Additionally, the distribution of $b$ is known 
to be a beta$(p/2,(n-p)/2)$ distribution. This follows from 
the fact that the squared elements of
a row of a random 
matrix uniformly distributed on $\mathcal O_n$ are 
jointly distributed Dirichlet$(1/2,\ldots, 
1/2)$.
We summarize these results with the following theorem:
\begin{theorem}
The uniformly most powerful invariant level-$\alpha$ test of 
$H:\omega=0$ versus $K:\omega>0$ in the rank-1 spiked covariance 
model is given by 
\[ \phi(\bl Y) = 1 ( \bl c^T \bl U \bl U^T \bl c  > b_{1-\alpha}),  \]
where 
$b_{1-\alpha}$ is the $1-\alpha$ quantile of a $\text{beta}(p/2,(n-p)/2)$ distribution. 
The power of this test is given by 
\begin{align*}
\Pr\left ( \frac{1+\omega}{1+b\omega} b > b_{1-\alpha} \right )  &= 
\Pr\left ( b > \frac{b_{1-\alpha}}{ 1+ \omega (1-b_{1-\alpha} ) } \right ) ,  
\end{align*} 
where $b \sim $ beta$(p/2,(n-p)/2))$. 
\label{thm:umpi}
\end{theorem}

\begin{figure}
\includegraphics[width=6in]{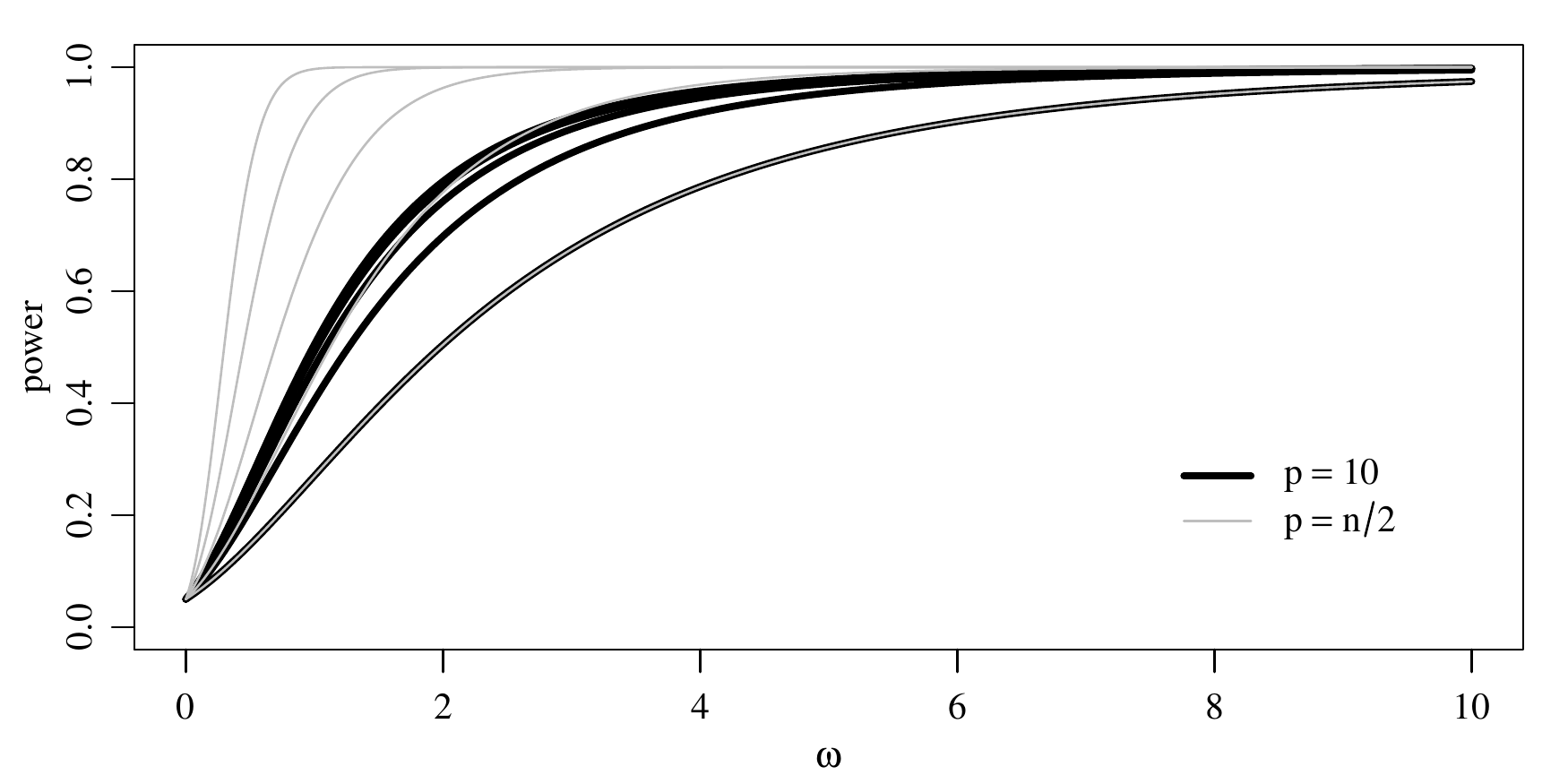} 
\caption{Power
of the level-0.05 UMPI test 
 as a function of $\omega$ for various $p$ and $n\in 
 \{ 20,40,80,160,320 \}$. }
\label{fig:sc1pow}
\end{figure}

The power of the level-0.05 test for various values of $p$ and $n$ 
are shown in Figure \ref{fig:sc1pow}. 
Note that the power does not go to one with increasing 
$n$ if $\omega$ and $p$ are fixed. This makes intuitive sense -
in this case the information per row is 
not increasing while the dimension of $\bl c$ is. 
However, it should be noted that the power 
for fixed 
$n$ and $\omega$ is non-monotonic in $p$: 
Some numerical calculations (not presented here) indicate
that the optimal power for 
moderate or large values of $n$ or $\omega$ is when $p\approx n/2$, 
and is somewhat less than this if  $n$ and $\omega$ are both small.

It is interesting to note that for this submodel, the likelihood 
ratio test is quite bad. Straightforward calculations 
show that the MLE of $\omega$ is  
$\hat \omega = \frac{ n t_{\bl c}  -p }{ p(1-t_{\bl c}) }$. 
Plugging this into the likelihood indicates that 
 a likelihood ratio test is one that rejects 
when $(n-p) \log (1-t_{\bl c}) +  p \log t_{\bl c}$ is large. 
This quantity is not monotonic in the UMPI test statistic 
$t_{\bl c}$, and performs poorly as a result. 

Finally, it is straightforward to extend Theorem \ref{thm:umpi} 
to the matrix normal regression model: Consider testing 
$H: \omega= 0$ versus $K: \omega>0$ based on 
$\bl Y \sim N_{n\times p}(\bl X \bl  B^T , \Sigma \otimes (  
  \omega \bl c \bl c^T + \bl I)) $. As shown in Section 2, 
any invariant test must depend on $\bl Y$ only through 
$\bl G(\tilde{\bl Y}) =  \tilde {\bl Y}( \tilde {\bl Y}^T \tilde {\bl Y})^{-1} \tilde {\bl Y}^T$, 
where $\tilde{\bl Y} = \bl H \bl Y$. 
Under the spiked model, $\tilde{\bl Y} \sim 
  N_{(n-q)\times p }( \bl 0 , \Sigma\otimes ( \bl I + \tilde \omega \tilde{\bl c}\tilde{\bl c}^T))$, where $\tilde{\bl c} = \bl H \bl c/||\bl H \bl c||$
and $\tilde \omega = \omega ||\bl H\bl c||^2 $. 
By Theorem \ref{thm:umpi}, 
the most powerful test of $H$ versus $K$ based on  
 $\bl G(\tilde{\bl Y})$, and hence the most powerful invariant test, 
is obtained by rejecting when 
$\tilde {\bl c}^T \bl G(\tilde{\bl Y})
 \tilde {\bl c}$ is large. This quantity  
can be expressed in more familiar forms as follows:
\begin{align*}
\tilde {\bl c}^T \bl G(\tilde{\bl Y})  \tilde {\bl c} &= 
 \bl c^T \bl H^T \left [ \tilde {\bl Y}( \tilde {\bl Y}^T \tilde {\bl Y})^{-1} \tilde {\bl Y}^T       \right ]  \bl H \bl c  \\ 
&=  \bl c^T \bl H^T \bl H \bl Y ( {\bl Y}^T  \bl H^T \bl H {\bl Y})^{-1} {\bl Y}^T     \bl H^T   \bl H \bl c  \\  
&= \bl c^T \bl R  (\bl R^T \bl R)^{-1} \bl R^T \bl c   = 
  \bl c^T \bl  G( \bl R)  \bl c . 
\end{align*}
Furthermore, this can also be expressed as  
$ \bl c^T  ( \bl Y - \bl X \hat{\bl B}^T) \hat \Sigma^{-1} 
            ( \bl Y - \bl X \hat{\bl B}^T)^T  \bl c /n$, 
where $\hat {\bl B }^T =  (\bl X^T\bl X)^{-1} \bl X^T \bl Y$ is 
the OLS estimate of $\bl B^T$, and 
$\hat \Sigma= \bl R^T \bl R/n$ is the MLE of $\Sigma$ under $H$. 
By Theorem \ref{thm:umpi}, this test statistic has a beta$( p/2, (n-q-p)/2)$
distribution under the null hypothesis. 

\section{A test of positive row dependence}
The UMPI test developed in the previous section 
is of limited applicability, as typically 
the space of alternatives of interest 
is larger 
than that provided 
by a spiked covariance model with a fixed eigenvector $\bl c$. 
However, the UMPI test suggests the possibility of 
constructing tests based
a set of statistics $t_{\mathcal C} = \{ t_{\bl c} = \bl c^T \bl G(\bl R) \bl c :  
\bl c \in \mathcal C\}$, 
where $\mathcal C \subset \mathbb R^{n}$ is a set of vectors 
of particular interest. 
For example, suppose there is concern that some rows 
of $\bl Y$ are positively correlated with each other. 
Based on the results of the previous section, 
the test statistic 
$t_{ii'} = \bl c_{ii'}^T \bl G(\bl R)  \bl c_{ii'}$ 
could be used
to detect positive correlation between rows $i$ and $i'$, 
where $\bl c_{ii'} = (\bl e_i + \bl e_{i'})/\sqrt{2}$ is the vector 
with entries of $1/\sqrt{2}$  in positions $i$ and $i'$ and 
entries of zero elsewhere.
However, if there is no information as to which rows might be 
correlated, some summary of the set of 
pairwise test statistics
$\{ t_{\bl c_{ii'}} = \bl c_{ii'}^T \bl G(\bl R)  \bl c_{ii'} :  
  i\neq i' \}$ could be used as a test statistic. 
Given a 
 residual matrix $\bl R$, 
the values of these test statistics can be computed quite 
easily: Some straightforward matrix calculations show that the 
value of $t_{\bl c_{ii'}}$ for $i\neq i'$ 
is given by element $(i,i')$ of the 
matrix $\bl T$, where 
\begin{equation}  
\bl T = \bl G(\bl R)  +  (\bl g \bl 1^T + \bl 1 \bl g^T) /2, 
\label{eqn:pms} 
\end{equation}
and $\bl g$ is the diagonal of $\bl G(\bl R)$. 

A test for positive dependence among pairs of rows of $\bl Y$ can be based
on a scalar summary function 
of the non-diagonal entries of $\bl T$. 
Letting $\tilde{t}$ be one such function, 
the null distribution of $\tilde{t}$ may be obtained via simulation, 
as the distribution of 
$\bl G(\bl R)$ 
does not depend on any unknown parameters under the null model. 
A Monte Carlo approximation to 
the null distribution of $\tilde t$ may be obtained 
via simulation of independent 
$n\times p$ random matrices $\bl Y^{(1)},\ldots, \bl Y^{(S)}$ 
with standard normal entries. 
For each simulated matrix  $\bl Y^{(s)}$
a residual matrix $\bl R^{(s)}$ is obtained as determined by the mean model. 
From $\bl R^{(s)}$, values of  
 $\bl G(\bl R^{(s)})$ , $\bl T^{(s)}$
 and  $\tilde t^{(s)}$ may be computed. 
The critical value for a level-$\alpha$ test based on the 
test statistic $\tilde t$ is approximated by the 1-$\alpha$ 
quantile of $\tilde t^{(1)},\ldots, \tilde t^{(S)}$.

The choice  of the summary function $\tilde t$ may depend on 
application-specific concerns about a particular type of 
dependence. Concern about dependence between  
small number of unspecified rows would suggest 
using 
the maximum of the off-diagonal elements of 
the matrix $\bl T$ in (\ref{eqn:pms}) as a test statistic. 
We refer to this statistic as $t_{\rm max}$, and the resulting 
test as the maxEP test (maximum exchangeable pair test). 
Figure \ref{fig:pmaxpow} shows the power of the 
level-0.05 maxEP test 
under 
the 
mean-zero model and 
alternative $\Psi = \bl I + \omega \bl c_{ii'} \bl c_{ii'}^T$
for a variety of sample sizes and  values of $\omega$
(the choice of $i$ or $i'$ does not affect the power). 
Note that if it were known in advance which pair of 
rows  $(i,i')$ were possibly 
correlated, the UMPI 
test statistic 
$t_{\bl c_{ii'}}$  
could be used, giving the power shown in 
Figure \ref{fig:sc1pow}. The difference between 
Figure \ref{fig:sc1pow} and Figure \ref{fig:pmaxpow} 
indicates the power loss that results
from considering the larger class of alternatives.

\begin{figure}
\includegraphics[width=6in]{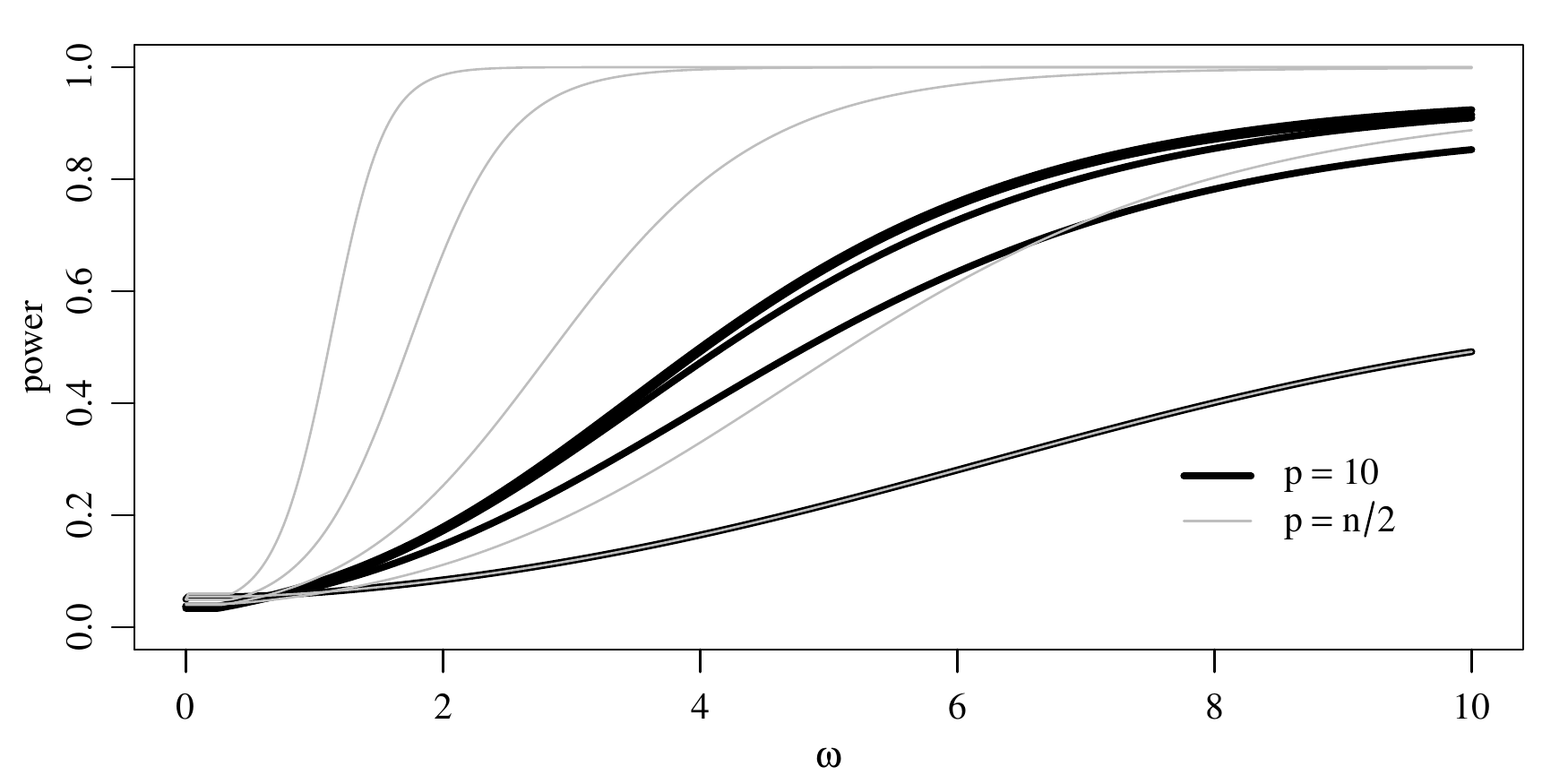}
\caption{Power
of the level-0.05 maxEP test
 as a function of $\omega$ for various $p$ and $n\in 
 \{ 20,40,80,160,320 \}$. }
\label{fig:pmaxpow}
\end{figure}

To illustrate its use,   the maxEP
test was applied to three datasets 
using  a few different mean models. 
The first dataset is described in  \citet{ashley_etal_2006} and
has been analyzed by \citet{efron_2009}, among others. 
The second two datasets are described more fully in \citet{flury_1997}.
For each test on each dataset, 
the null distribution was approximated by a Monte Carlo 
sample of size 5,000. The computer code for implementing 
these tests is available at my website, 
\url{http://www.stat.washington.edu/~pdhoff/}.

\paragraph{Cardio:} 
This  dataset consists of 20,426 gene expression levels 
measured on $n=63$ subjects. 
Although 20,426 gene expression variables are available, any invariant test must be a function of
less than 63 of these. Based on the discussion  of power 
that followed Theorem \ref{thm:umpi}, 
only the first $p=32\approx n/2$ variables were used to perform the test. 
As in  \citet{efron_2009}, inference is based on 
a doubly-centered residual matrix 
$\bl R$ obtained by de-meaning the rows and columns of the data matrix $\bl Y$, 
so that $\bl R=(\bl I_n-\bl 1_n\bl 1_n^T/n ) \bl Y (\bl I_p- \bl 1_p\bl 1_p^T/p )$.
The observed value of  $t_{\rm max}$ based on 
$\bl R$ 
is .927.  
In contrast, the largest value of $t_{\rm max}$ observed in 
the Monte Carlo sample
was 0.856, giving an approximate Monte Carlo $p$-value of zero
and indicating
strong evidence against the null model.

\paragraph{Turtles:} 
These data consist of length, width and height measurements of
24 male and 24 female turtles, sampled from a single pond on a
single day. 
Two tests were applied  to the 
log-transformed data, the first of which 
tested $H:\Psi=\bl I$ versus 
$K:\Psi\neq \bl I$ 
in the column means model, 
$\bl Y \sim N_{n\times p}( \bl 1 \bs \mu^T , \Sigma\otimes \Psi)$, 
so that under the null model the rows of $\bl Y$ are i.i.d.\ 
$p$-variate normal random vectors. 
The residual matrix for this mean model is 
$\bl R = (\bl I_n - \bl 1_n \bl 1_n^T/n) \bl Y$, 
which gives an
observed $t_{\rm max}$ statistic of 0.344 and a 
$p$-value of 0.16. 
The second test is 
based on the matrix normal regression model 
(\ref{eqn:mnm}) where $\bl X$ is the $n\times 2$ matrix indicating 
the sex of each turtle. The residual matrix here is 
$\bl R = (\bl I  - \bl X (\bl X^T \bl X)^{-1} \bl X^T) \bl Y$, 
which gives an observed test statistic of 
$t_{\rm max}= 0.336$, corresponding to a $p$-value of 
0.20.

\paragraph{Wines:} 
These data consist of measurements of $p=15$ organic compounds on $n=26$ Riesling wines.
Tests were applied to the log-transformed data.
The wines were selected from different vintners from three countries,
and do not constitute a random sample. 
Evidence of row covariance was evaluated 
in the context of 
the same 
mean models as for the
turtle data - a column means model and a model 
taking into account a known categorical variable. 
For the column means model, 
the $t_{\rm max}$ statistic and 
the $p$-value for the maxEP test
were 0.893 and 
0.007 respectively, indicating strong evidence against 
the null model of i.i.d.\ measurements. 
However, 
after accounting for country differences 
via the matrix normal regression model 
(with $\bl X$  being the  $n\times 3$ matrix indicating country of origin), 
the test statistic and $p$-value
were 0.843 and 0.23 respectively, 
indicating little evidence against 
$H:\Psi = \bl I$ after accounting for mean differences due to country.

\section{Discussion} 
The results of this article were developed in the 
context of a matrix normal error variance model, 
but they hold more generally for models with 
stochastic representations of the form 
$\bl Y = \bl X \bl B^T +  \Psi^{1/2} \bl Z \Sigma^{1/2} $. 
For example, 
the characterization of the maximal invariant statistics 
in Section 2 relies only on the invariance of the model 
and that $\bl Z$ is  full rank with probability one. 
The results of Sections 3 and 4 depend only on the distribution 
of the maximal invariant statistic, 
which in turn depends on $\bl Z$ only through 
$\bl W = \bl Z (\bl Z^T \bl Z)^{-1/2}$. 
For a normal error variance model the distribution of 
$\bl W$ is uniform on the Stiefel manifold, 
but this is also true for 
any  
model where the distribution of the  vectorization of $\bl Z$ 
is spherically symmetric. 
The class of models for $\bl Y$ 
in which $\text{vec}(\bl Z)$ is spherically symmetric are the 
elliptically contoured matrix distributions 
\citep{gupta_varga_1994}, which includes
heavy-tailed and contaminated distributions, among others. 

This article has considered tests of $H: \Psi=\bl I$ versus $K:\Psi\neq \bl I$, 
that is, tests of whether or not the 
rows of the error matrix $\bl Y - \bl X \bl B^T$
are independent \emph{and} identically distributed. 
This null hypothesis is violated not just when the rows are dependent, 
but also when they are 
heteroscedastic and 
independent. 
However, in some applications it may be useful to have 
a test that includes independent heteroscedasticity as part of the 
null hypothesis.  
\citet{volfovsky_hoff_2015} studied a likelihood 
ratio test of $H: (\Sigma ,\Psi) \in \mathcal D_{p}^+ \times 
\mathcal D_n^+$
versus $K: (\Sigma ,\Psi) \notin \mathcal D_{p}^+ \times 
\mathcal D_n^+$,  where
$\mathcal D_{k}^+$ is the set of $k\times k$ diagonal matrices 
with positive entries.   
However, their test is only applicable to square data matrices, 
and will reject in the presence of either row or column dependence. 
For testing $H: \Psi \in \mathcal D_n^+$
versus
$K: \Psi \not\in \mathcal D_n^+$ 
it might be possible to use invariance, 
 but perhaps not directly:
A natural group with which to find 
an invariant procedure are the transformations of 
the form $g(\bl Y) = \bl A \bl Y \bl B^T $, 
where $\bl A\in \mathcal D_n^+$ and $\bl B\in \mathbb R^{p\times p}$ 
is nonsingular. 
However, while the covariance model is invariant to such transformations
the mean model is not, and so it seems that 
to usefully apply invariance one would first need to reduce 
to a mean-zero model, as was 
done  in Section 2.3 for mean models with row effects.